\documentclass[11pt,a4paper]{article}
\usepackage{amsmath,amssymb,amsthm,graphicx}
\usepackage{hyperref}
\advance\textwidth by 2,93mm
\advance\textheight by 3,9mm

\newtheorem{theorem}{Theorem}
\newtheorem{lemma}{Lemma}
\newtheorem{claim}{Claim}
\newtheorem{proposition}{Proposition}

\theoremstyle{definition}
\newtheorem{definition}{Definition}
\newtheorem{problem}{Problem}

\newcommand{\conv}[1]{{\mbox{conv}(#1)}}

\newcommand\A{\mathcal{A}}
\newcommand\LL{L}

\begin{document}
\title{On the maximum size of an anti-chain
 of linearly separable sets
 and convex pseudo-discs$^1$}
\author{
   Rom Pinchasi\thanks{
      Mathematics Dept.,
      Technion---Israel Institute of Technology,
      Haifa 32000, Israel.
\hfil\penalty0\hfilneg
      \texttt{room@math.technion.ac.il}.
   } \and
   G\"unter Rote\thanks{
      Institut f\"ur Informatik,
      Freie Universit\"at Berlin, Takustr.\ 9, 14195 Berlin, Germany.
\hfil\penalty0\hfilneg
      \texttt{rote@inf.fu-berlin.de}}
}
\footnotetext[1]{This research was supported by a Grant from the G.I.F., the
German-Israeli Foundation for Scientific Research and
Development.}
\maketitle

\begin{abstract}
  We answer a question raised by Walter Morris, 
  and independently by Alon Efrat, about the maximum cardinality
  of an anti-chain composed of intersections of a given set of $n$
  points in the plane with half-planes. We approach this problem by
  establishing the equivalence with the problem of the maximum
  monotone path in an arrangement of $n$ lines. A related problem on
  convex pseudo-discs is also discussed in the paper.
\end{abstract}


\section{Introduction}

Let $P$ be a set of $n$ points in the plane, no three of which are collinear.
A subset of $P$ is called \emph{linearly separable} if it is the intersection
of $P$ with a closed half-plane.
A $k$-set of $P$ is a subset of $k$ points from $P$ which is linearly 
separable.
Let $\A_{k}=\A_k(P)$ denote the collection
of all $k$-sets of $P$. It is a well-known open problem to determine $f(k)$, 
the maximum possible cardinality of $\A_{k}$, where $P$ varies over all possible
sets of $n$ points in general position in the plane. The current records are
$f(k)=O(nk^{1/3})$ by Dey (\cite{D98}) and 
$f(\lfloor n/2 \rfloor) \geq ne^{\Omega(\sqrt{\log n})}$
by T\'oth (\cite{T01}).

Let $\A=\A(P)=
\cup_{k=0}^{n}\A_{k}$ be the family of all linearly separable 
subsets of $P$.
The family $\A$ is partially ordered by inclusion. 
Clearly, each $\A_{k}$ is an anti-chain in 
$\A$. The following problem was raised by Walter Morris in 2003 in relation 
with the \emph{convex dimension} of a point set (see \cite{ES88}) and,
as it turns out, it was independently raised by Alon Efrat 10 years before, 
in 1993: 

\begin{problem}\label{problem-1}
What is the maximum possible cardinality $g(n)$ 
of an anti-chain in the poset $\A$,
over all sets $P$ with $n$ points?
\end{problem}

In Section \ref{section:antichain} we show that in fact $g(n)$ can be 
very large, and in particular much larger than $f(n)$.

\begin{theorem}\label{theorem:antichain}
$g(n) =
\Omega(n^{2-\frac{d}{\sqrt{\log n}}})$, for some absolute constant $d>0$.
\end{theorem} 


In an attempt to bound from above the function $g(n)$ one can view linearly
separable sets
as a special case of a slightly more general concept:

\begin{definition}
Let $P$ be a set of $n$ points in general position in the plane.
A Family $F$ of subsets of $P$ is called a family of 
\emph{convex pseudo-discs} if
the following two conditions are satisfied:

\begin{enumerate}

\item Every set in $F$ is the intersection of $P$ with a convex set.  

\item If $A$ and $B$ are two different sets in $F$,
then both sets $\conv{A} \setminus \conv{B}$ and 
$\conv{B} \setminus \conv{A}$ are connected (or empty).

\end{enumerate}
\end{definition}

One natural example for a family of convex pseudo-discs is the family $\A(P)$,
where $P$ is a set of $n$ points in general position in the plane.
To see this, observe that every linearly separable set
is the intersection of $P$ with a 
convex set, namely, a half-plane. It is therefore left to verify that if
$A=P \cap H_{A}$ and $B=P \cap H_{B}$, where $H_{A}$ and $H_{B}$ are
two half-planes, then both $\conv{A} \setminus \conv{B}$ and 
$\conv{B} \setminus \conv{A}$ are connected.
Let $A'=A \setminus H_{B}=A \setminus B = A\setminus \conv{B}$.  Since
$\conv{A'} \cap \conv{B} =\emptyset$, we have $\conv{A} \setminus
\conv{B} \supset \conv{A'}$. For any $x \in \conv{A} \setminus
\conv{B}$, we claim that there is a point $a' \in A'$ such that the
line segment $[x,a']$ is fully contained in $\conv{A} \setminus
\conv{B}$. This will clearly show that $\conv{A} \setminus \conv{B}$
is connected. Let $a_{1}, a_{2},a_{3}$ be three points in $A$ such
that $x$ is contained in the triangle 
$a_{1}a_{2}a_{3}$.  If each line segment $[x,a_{i}]$, for $i=1,2,3$,
contains a point of $\conv{B}$, it follows that $x \in \conv{B}$,
contrary to our assumption. Thus there must be a line segment $[x,a_{i}]$ that is contained in $A'=A\setminus \conv{B}$, and we are done.

In Section \ref{section:n2} we bound from above the maximum size of a family
of convex pseudo-discs of a set $P$ of $n$ points in the plane, assuming
that this family of subsets of $P$ is by itself an anti-chain with respect
to inclusion:

\begin{theorem}\label{theorem:n2}
Let $F$ be a family of convex pseudo-discs of a set $P$ of $n$
points in general position in the plane. If no member of $F$ is contained 
in another, then $F$ consists of at most
$4\binom n 2+1$ members. 
\end{theorem}

Clearly, in view of Theorem \ref{theorem:antichain}, 
the result in Theorem \ref{theorem:n2} is nearly best possible. We show by a
simple construction that Theorem \ref{theorem:n2} is in fact tight, apart from 
the constant multiplicative factor of $n^2$.

\section{Large anti-chains of linearly separable sets}\label{section:antichain}

Instead of considering Problem~\ref{problem-1} directly, we will consider a
related problem.

\begin{definition}
  For a pair $x,y$ of points and a pair
$\ell_{1}, \ell_{2}$ of non-vertical lines,
we say that
 $x,y$ \emph{strongly separate}
$\ell_{1}, \ell_{2}$ if
$x$ lies strictly above $\ell_{1}$ and strictly below $\ell_{2}$, and
$y$ lies strictly above $\ell_{2}$ and strictly below $\ell_{1}$.

We will also take the dual viewpoint and say that $\ell_{1}, \ell_{2}$
strongly separate $x,y$.  (In fact, this relation is invariant under
the standard point-line duality.)

If we have a set $\LL$ of lines, we say that the point pair
$x,y$ is \emph{strongly separated} by $\LL$, if $\LL$ contains two lines
$\ell_{1}, \ell_{2}$ that strongly separate $x,y$.

A pair of lines $\ell_{1}, \ell_{2}$ is said to be strongly
separated by a set $P$ of points if there are two points $x,y\in P$
that strongly separate $\ell_{1}$ and $\ell_{2}$.
\end{definition}

Using the above terminology one can reduce Problem~\ref{problem-1} to the following
problem:

\begin{problem}\label{problem-2}
Let $P$ be a set of $n$ points in the plane.
What is the maximum possible cardinality $h(n)$ (taken over all
possible sets $P$ of $n$ points)
of a set of lines $\LL$ in the plane such 
that for every two lines $\ell_{1}, \ell_{2} \in \LL$,
 $P$ strongly 
separates $\ell_{1}$ and $\ell_{2}$.
\end{problem}

\begin{figure}[ht]
\begin{center}
\includegraphics{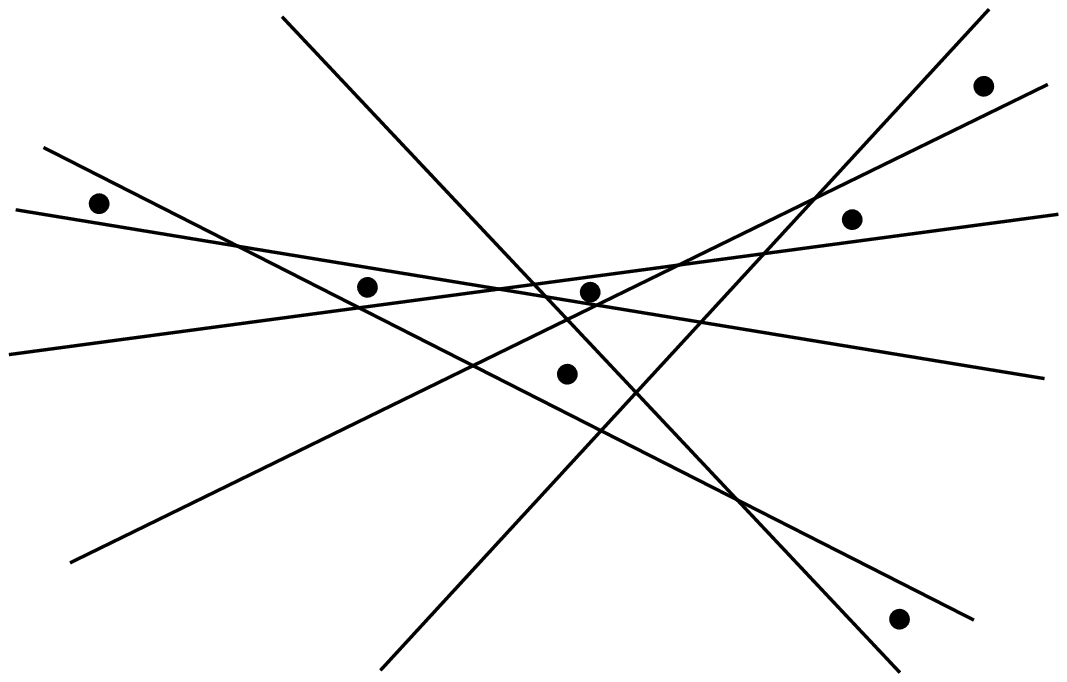}
\caption{Problem~\ref{problem-2}.}
\label{fig:1}
\end{center}
\end{figure}

To see the equivalence of Problem~\ref{problem-1} and Problem~\ref{problem-2}, 
let $P$ be a set 
of $n$ points and $\LL$ be a set of $h(n)$ lines that answer 
Problem~\ref{problem-2}.
We can assume that none of the points lie on a line of $\LL$.
 Then with each of the lines $\ell \in \LL$ we associate the subset
of $P$ which is the intersection of $P$ with the half-plane
 below $\ell$. We thus obtain $h(n)$ subsets of $P$ each
of which is a linearly separable subset of $P$. 
Because of the condition on $\LL$ and $P$,
none of these linearly separable sets may contain another. 
Therefore we obtain $h(n)$
elements from $\A(P)$ that form an anti-chain, hence $g(n) \geq h(n)$.

Conversely, assume we have an anti-chain of size $g(n)$ in 
$\A(P)$ for a set $P$ of $n$ points.
Each linearly separable set is the intersection of $P$ with a half-plane, 
which is
bounded by some line $\ell$. We can assume without loss of generality
that
none of these lines is vertical, and at least half of the half-spaces
lie
below their bounding lines. These lines form a set $L$ of at least
$\lceil g(n)/2\rceil$ lines, and each pair of lines is separated by
two points from the $n$-point set $P$. Thus,
$h(n) \geq \lceil g(n)/2 \rceil$.

Before reducing Problem~\ref{problem-2} to another problem, we need the following simple
lemma.

\begin{lemma}\label{lemma:consecutive}
Let $\ell_{1}, \ldots, \ell_{n}$ be $n$ non-vertical lines arranged in
increasing order of slopes. Let $P$ be a set of points.
Assume that for every $1 \leq i <n$, $P$
strongly separates $\ell_{i}$ and $\ell_{i+1}$. 
Then for every $1 \leq i < j \leq n$,
$P$ strongly separates $\ell_{i}$ and $\ell_{j}$.
\end{lemma}

\begin{proof}
We prove the lemma by induction on $j-i$.
For $j=i+1$ there is nothing to prove. Assume $j-i \geq 2$.
We first show the existence of a point $x \in P$ that lies above $\ell_{i}$ 
and below $\ell_{j}$. Let $B$ denote the intersection point of $\ell_{i}$
and $\ell_{j}$. Let $r_i$ denote the ray whose apex is $B$, included in 
$\ell_{i}$, and points to the right. Similarly, let $r_{j}$ denote the ray
whose apex is $B$, included in $\ell_{j}$, and points to the right.

Since the slope of $\ell_{i+1}$ is between the slope of $\ell_{i}$ and 
 the slope of $\ell_{j}$, $\ell_{i+1}$ must intersect either 
$r_{i}$ or $r_{j}$ (or both, in case it goes through~$B$).

\noindent {\bf Case 1.} $\ell_{i+1}$ intersects $r_{i}$. Then there is a point 
$x \in P$ that lies above $\ell_{i}$ and below $\ell_{i+1}$. This point $x$
must also lie below $\ell_{j}$.

\noindent {\bf Case 2.} $\ell_{i+1}$ intersects $r_{j}$. Then, by the 
induction hypothesis, there is a point 
$x \in P$ that lies above $\ell_{i+1}$ and below $\ell_{j}$. This point $x$
must also lie above $\ell_{i}$.

The existence of a point $y$ that lies above $\ell_{j}$ and below $\ell_{i}$
is symmetric.
\end{proof}

By Lemma \ref{lemma:consecutive},  Problem~\ref{problem-2} is equivalent to
 following problem.

\begin{problem}\label{problem-3}
  What is the maximum cardinality $h(n)$ of a collection of lines
  $\LL=\{\ell_{1}, \ldots ,\ell_{h(n)}\}$ in the plane, indexed so
  that the slope of $\ell_{i}$ is smaller than the slope of $\ell_{j}$
  whenever $i<j$, such that there exists a set $P$ of $n$
  points 
  that strongly separates $\ell_{i}$ and $\ell_{i+1}$, for every $1
  \leq i < h(n)$?
\end{problem}

We will consider the dual problem of Problem~\ref{problem-3}:

\begin{problem}\label{problem-4}
What is the maximum cardinality $h(n)$ of a set of points
$P=\{p_1, \ldots, p_{h(n)}\}$ in the plane, indexed so that 
the $x$-coordinate of $p_{i}$ is smaller than the $x$-coordinate of $p_{j}$,
whenever $i<j$, 
such that
there exists a set 
$\LL$ of $n$ lines that strongly separates $p_{i+1}$ and $p_{i}$,
 for every $1 \leq i < h(n)$?
\end{problem}

We will relate Problem~4 to another well-known problem:
the question of the longest monotone path in an arrangement of lines.

Consider an $x$-monotone path in a
line arrangement in the plane.
The \emph{length} of such a path is the number of different line segments
that constitute the path, assuming that consecutive line segments on the path 
belong to different lines in the arrangement. (In other words, if the
path passes through a vertex of the arrangement without making a turn,
this does not count as a new edge.)

\begin{problem}\label{problem-monotone}
  What is the maximum possible length $\lambda(n)$ of an $x$-monotone
path in an arrangement of $n$ lines?
\end{problem}

A construction of \cite{BRSSS04} gives a simple line arrangement in
the plane which consists of $n$ lines and which contains an
$x$-monotone path of length $\Omega(n^{2-\frac{d}{\sqrt{\log n}}})$
for some absolute constant $d>0$.
No upper bound that is asymptotically better than the trivial bound
of $O(n^2)$ is known.

Problem~\ref{problem-monotone}
is closely related to Problem~\ref{problem-4}, and hence also to the other problems:

\begin{proposition}
  \begin{equation}
    \label{eq:h1}
h(n) \geq \left\lceil\frac{\lambda(n)+1}{2}\right\rceil,
  \end{equation}
  \begin{equation}
    \label{eq:h2}
\lambda(n) \geq h(n)-2
  \end{equation}
\end{proposition}

\begin{proof}
  We first prove
$h(n) \geq \lceil(\lambda(n)+1)/2\rceil$. Let $\LL$ be a simple 
arrangement of $n$ lines that admits an
$x$-monotone path of length $m=\lambda(n)$.
Denote by $x_0,x_1, \ldots, x_{m}$ the vertices of a monotone path
arranged in increasing order of $x$-coordinates. In this notation 
$x_1, \ldots, x_{m-1}$ are vertices of the line arrangement $\LL$, while
$x_0$ and $x_m$ are chosen arbitrarily on the corresponding two rays which 
constitute the first and last edges, respectively, of the path.
For each $1 \leq i < m$ 
let $s_{i}$ denote the line that contains the segment
 $x_{i-1}x_{i}$, and let $r_{i}$ denote the line through
 the segment $x_{i}x_{i+1}$.

For $1 \leq i < m$, we say that the path bends downward at the vertex 
$x_{i}$ if the slope of 
$s_{i}$ is greater than the slope of $r_{i}$,
and it bends upward if the 
slope of $s_{i}$ is smaller than the slope of $r_{i}$.
Without loss of generality we may assume that at least half of the vertices 
$x_{1}, \ldots, x_{m-1}$ of the monotone path are downward bends.

\begin{figure}[ht]
\begin{center}
\includegraphics
{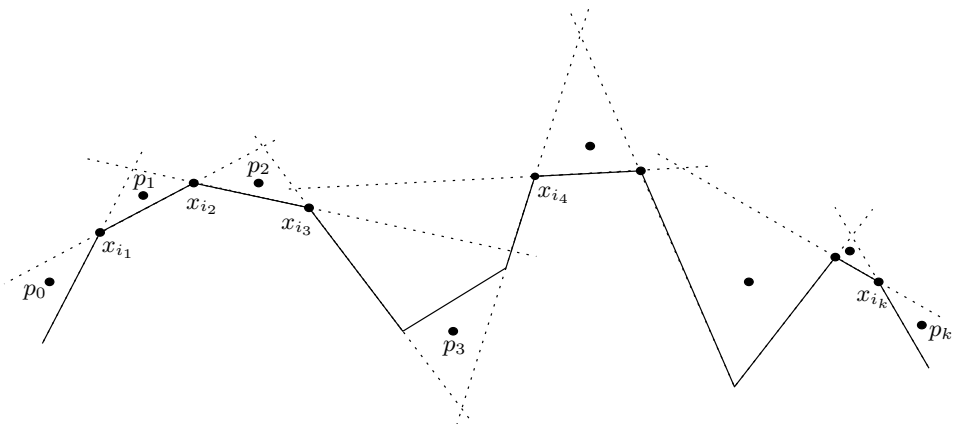}
\caption{Constructing a solution for Problem~\ref{problem-4}.}
\label{fig:2}
\end{center}
\end{figure}

Let $i_{1}< i_{2} < \dots < i_{k}$ be all indices such that $x_{i_{j}}$
is a downward bend, where $k\ge (m-1)/2$. 
Observe that for every $1 \leq j < k$,
the monotone path between $x_{i_{j}}$ and $x_{i_{j+1}}$ is an
upward-bending convex polygonal path.

We will now define $k+1$ points $p_{0}, p_{1}, \ldots, p_{k}$ such that
for every $0 \leq j <k$ the $x$-coordinate of $p_{j}$ is smaller than the
$x$-coordinate of $p_{j+1}$, and the line $r_{i_{j}}$ lies above $p_{j}$
and below $p_{j+1}$ while the line $s_{i_{j}}$ lies below $p_{j}$ and
above $p_{j+1}$. This construction will thus show that 
$h(n) \geq \lceil\frac{\lambda(n)+1}{2}\rceil$.

For every $1 \leq j \leq k$ let $U_{j}$ and $W_{j}$ denote the left and 
respectively the right wedges delimited by $r_{i_{j}}$ and $s_{i_{j}}$.
That is, $U_{j}$ is the set of all points that lie below $r_{i_{j}}$ and
above $s_{i_{j}}$.
Similarly, $W_{j}$ is the set of all points that lie above $r_{i_{j}}$ and
below $s_{i_{j}}$.

\begin{claim}\label{lemma:intersect}
For every $1 \leq j < k$, $W_{j}$ and $U_{j+1}$ have a nonempty intersection.
\end{claim}

\begin{proof}
We consider two possible cases:

\noindent{\bf Case 1.} $i_{j+1}=i_{j}+1$. In this case $r_{i_{j}}=s_{i_{j+1}}$.
Therefore any point above the line segment $[x_{i_{j}}x_{i_{j+1}}]$ that is
close enough to that segment lies both below $s_{i_{j}}$ and below 
$r_{i_{j+1}}$ and hence $W_{j} \cap U_{j+1} \neq \emptyset$.

\noindent{\bf Case 2.} $i_{j+1}-i_{j}>1$. In this case, as we observed earlier,
the monotone path between $x_{i_{j}}$ and $x_{i_{j+1}}$ is a convex 
polygonal path. Therefore, $r_{i_{j}}$ and $s_{i_{j+1}}$ are different lines
that meet at a point $B$ whose $x$-coordinate is between the $x$-coordinates
of $x_{i_{j}}$ and $x_{i_{j+1}}$.
Any point that lies vertically above $B$ and close enough to $B$
belongs to both $W_{j}$ and $U_{j+1}$. 
\qed

Now it is very easy to construct $p_{0}, p_{1}, \ldots ,p_{k}$,
see Figure~\ref{fig:2}.
Simply take $p_{0}$ to be any point in $U_{1}$, and for every $1 \leq j <k$
let $p_{j}$ be any point in $W_{j} \cap U_{j+1}$. Finally, let $p_{k}$ be any
point in $W_{k}$. It follows from the definition of 
$U_{1}, \ldots, U_{k}$ and $W_{1}, \ldots, W_{k}$ that for every 
$0 \leq j < k$,
 $r_{i_{j+1}}$ lies above $p_{j}$ and below $p_{j+1}$ and the
line $s_{i_{j+1}}$ lies below $p_{j}$ and above $p_{j+1}$.

  We now prove the opposite direction:
$\lambda(n) \geq h(n)-2$.

Assume we are given $h(n)$ points $p_1, \ldots, p_{h(n)}$ sorted by 
$x$-coordinate 
and a set of $n$ lines $\LL$ such that every pair $p_i,p_{i+1}$ is
strongly separated by~$\LL$.
By perturbing the lines if necessary, we can assume that none of the
lines goes through a point, and no three lines are concurrent.
For $1<i<h(n)$,
let $f_i$ be the face of the arrangement that contains $p_i$, and
let $A_i$ and $B_i$ be, respectively, the left-most and right-most vertex in 
this face.
(The faces $f_i$ are bounded, and therefore $A_i$ and $B_i$ are well-defined.)
The monotone path will follow the upper boundary of each face $f_i$
from $A_i$ to $B_i$.

We have to show that we can connect $B_i$ to $A_{i+1}$ by a monotone
path.
This follows from the separation property of $L$.
Let $s_i,r_i$ be a pair of lines that strongly separates
 $p_i$ and $p_{i+1}$ in such a way that
 $r_i$ lies above $p_{i}$ and below $p_{i+1}$ and 
 $s_{i}$ lies below $p_{i}$ and above $p_{i+1}$.
Since $B_i$ lies on the boundary of the face $f_i$ that contains
$p_i$, $B_i$ lies also between $r_i$ and $s_i$, including the
possibility
of lying on these lines. We can thus walk on the arrangement
from $B_i$ to the right until we hit
 $r_i$ or $s_i$, and from there we proceed straight
to the intersection point $Q_i$ of $r_i$ and $s_i$.
Similarly, there is a path in the arrangement from $A_{i+1}$
to the left that reaches $Q_i$. 
and these two paths together link $B_i$ with $A_{i+1}$.

To count the number of edges of this path,
we claim that there must be at least one bend
between $B_i$ and $A_{i+1}$ (including
the boundary points $B_i$ and $A_{i+1}$).
If there is no bend at $Q_i$, the path must go straight through
$Q_i$, say, on $r_i$. But then the path must leave $r_i$
at some point when going to the right: if the path has
not left $r_i$ by the time it reaches $A_{i+1}$ and
$A_{i+1}$ lies on $r_i$, then the path must bend upward at this point,
since it proceeds on the upper boundary of the face $f_{i+1}$ that
lies
above $r_i$.

Thus, the path makes at least $h(n)-3$ bends
(between $B_i$ and $A_{i+1}$, for $1<i<h(n)-1$) and contains at least
$h(n)-2$ edges.
\end{proof}

Now it is very easy to give a lower bound for $g(n)$, and prove Theorem
\ref{theorem:antichain}.
Indeed, this follows because $g(n) \geq h(n)$ and 
$h(n) \geq \lceil\frac{\lambda(n)+1}{2}\rceil=
\Omega(n^{2-\frac{d}{\sqrt{\log n}}})$,

The close relation between Problems~\ref{problem-1} and~\ref{problem-monotone} comes
probably as no big surprise if one considers the close connection
between $k$-sets and \emph{levels} in arrangements of lines
(see \cite[Section 3.2]{Edelbook}).
For a given set of $n$
points $P$, the $k$-sets are in one-to-one correspondence with the
faces of the dual arrangements of lines which have $k$ lines passing
below them and $n-k$ lines passing above them (or vice versa).  The
lower boundaries of these cells form the $k$-th level in the
arrangement, and the upper boundaries form the $(k+1)$-st level.

Our chain of equivalence from Problem~\ref{problem-1} to
Problem~\ref{problem-monotone} extends this relation between $k$-sets
and {levels} in a way that is not entirely trivial: for example, establishing
that we get sets that form an antichain requires some work, whereas
for $k$-sets this property is fulfilled automatically.

\section{Proof of Theorem \ref{theorem:n2}}\label{section:n2}

The heart of our argument uses a linear algebra approach 
first applied by Tverberg~\cite{T82} in his elegant proof for
a theorem of Graham and Pollak~\cite{GP72} on decomposition of the complete
graph into bipartite graphs.

Let $F$ be a collection of convex pseudo-discs of a set $P$ of $n$ points
in general position in the plane. We wish to bound from above the size of
$F$ assuming that no set in $F$ contains another.
For every directed line $L=\overrightarrow{xy}$ passing through two points 
$x$ and $y$ in $P$ we denote
by $L_{x}$ the collection of all sets $A \in F$ that lie in the closed 
half-plane to the left of $L$ such that $L$ touches $\conv{A}$ at the point 
$x$ only. Similarly, let $L_{y}$ be the collection of all 
sets $A \in F$ that lie in the closed half-plane to the left of $L$ such that 
$L$ touches $\conv{A}$ at the point $y$ only. Finally, let $L_{xy}$ be those 
sets $A \in F$ that lie in the closed half-plane to the left of $L$ such that
$L$ supports $\conv{A}$ at the edge $xy$.

\begin{definition}
Let $A$ and $B$ be two sets in $F$. Let $L$ be a directed line through 
two points $x$ and $y$ in $P$. We say that $L$ is a common tangent of the 
\emph{first kind} 
with respect the pair $(A,B)$ if $A \in L_{x}$ and $B \in L_{y}$.

We say that $L$ is a common tangent 
of the second kind with respect to $(A,B)$ if
$A \in L_{xy}$ and $B \in L_{y}$, or if
 $A \in L_{x}$ and $B \in L_{xy}$.
\end{definition}

The crucial observation about any two sets $A$ and $B$ in $F$ 
is stated in the following lemma.

\begin{lemma}\label{lemma:crucial}
Let $A$ and $B$ be two sets in $F$.
Then exactly one of the following conditions is true.
\begin{enumerate}
\item There is precisely one common tangent
of the first kind with respect to $(A,B)$
and no common tangent of the second kind with respect to $(A,B)$,
or
\item there is no common tangent of the first kind with respect to $(A,B)$, and
there are precisely two common tangents of the second kind with respect $(A,B)$.
\end{enumerate}

\end{lemma}
\begin{figure}
  \centering
  \includegraphics{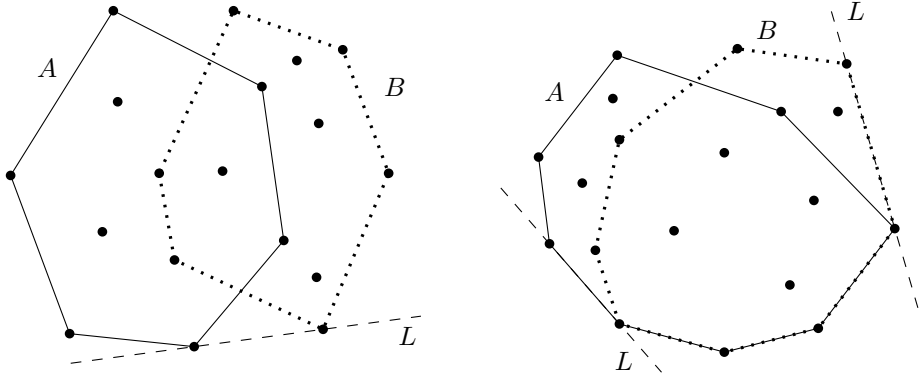}
  \caption{The two cases of common tangents in Lemma~\ref{lemma:crucial}}
  \label{common-tangents}
\end{figure}
\begin{proof}
  The idea is that because $A$ and $B$ are two pseudo-discs and none of 
$\conv{A}$ and $\conv{B}$ contains the other, then
as we roll a tangent around $C=\conv{A \cup B}$,
  there is precisely one transition between $A$ and $B$, and this is
  where the situation described in the lemma occurs (see
  Figure~\ref{common-tangents}).   

Formally, by our assumption on $F$, none of $A$
  and $B$ contains the other.  Any directed line $L$ that is a common
  tangent of the first or second kind with respect to $A$ and $B$ must
  be a line supporting $\conv{A \cup B}$ at an edge.

  Let $x_{0}, \ldots, x_{k-1}$ denote the vertices of $C=\conv{A \cup
    B}$ arranged in counterclockwise order on the boundary of $C$.  In
  what follows, arithmetic on indices is done modulo $k$.

There must be an index $i$ such that $x_{i} \in A \setminus B$, for
otherwise every $x_{i}$ belongs to $B$ and therefore 
$\conv{B} = \conv{A \cup B} \supset \conv{A}$ and therefore $B \supset A$ 
(because both $A$ and $B$ are intersections of $P$ with convex sets) 
in contrast to our assumption.
Similarly, there must be an index $i$ such that 
$x_{i} \in B \setminus A$. 

Let $I_{A}$ be the set of all indices $i$ such that $x_{i} \in A \setminus B$,
and let $I_{B}$ be the set of all indices $i$ such that 
$x_{i} \in B \setminus A$. 

We claim that $I_{A}$ (and similarly $I_{B}$) is a set of consecutive indices.
To see this, assume to the contrary that there are 
indices $i,j,i',j'$ arranged in a cyclic
order modulo $k$ such that $x_{i},x_{i'} \in A \setminus B$ and 
$x_{j},x_{j'} \in B$. Then it is easy to see that $\conv{A} \setminus \conv{B}$
is not a connected set because $x_{i}$ and $x_{i'}$ are in different connected
components of this set.

We have therefore two disjoint intervals $I_{A}=\{i_{A}, i_{A}+1,
\ldots, j_{A}\}$ and $I_{B}=\{i_{B}, i_{B}+1, \ldots, j_{B}\}$. It is
possible that $i_{A}=j_{A}$ or $i_{B}=j_{B}$.

Observe that $x_{i_{A}},x_{j_{A}},x_{i_{B}},x_{j_{B}}$ are arranged in this 
counterclockwise cyclic order on the boundary of $C$,
and for every index $i \notin I_{A} \cup I_{B}$,
$x_i \in A \cap B$. The only candidates for
common tangents of the first kind or of the second kind with respect to $A$ and $B$ are
of the form $\overrightarrow{x_{i}x_{i+1}}$, that is, they must pass through two
consecutive vertices of $C$.

We distinguish two possible cases:

\begin{enumerate}
\item $i_{B}=j_{A}+1$. In this case the line through
$x_{j_{A}}$ and $x_{i_{B}}$ is the only 
common tangent of the first kind with respect to $(A,B)$ and there are no
common tangents of the second kind with respect to $(A,B)$.

\item $i_{B} \neq j_{A}+1$. In this case,
there is no common tangent of the first kind with respect to $(A,B)$.
The line through $x_{i_{B}-1}$ and $x_{i_{B}}$ and the line
through $x_{j_{A}}$ and $x_{j_{A}+1}$ are the only common tangents of the second kind
with respect to $(A,B)$.
\end{enumerate}
This completes the proof of the lemma.
\end{proof}

Let $A_{1}, \ldots ,A_{m}$ be all the sets in $F$, and for every 
$1 \leq i \leq m$ let $z_{i}$ be an indeterminate associated with $A_{i}$.  
For each directed line $L=\overrightarrow{xy}$, define the following 
polynomial $P_{L}$: 
\begin{multline*}
P_{L}(z_{1}, \ldots,z_{m})=
\\ 
\biggl(\sum_{A_{i} \in L_{x}}z_{i}\biggr)\biggl(\sum_{A_{j} \in L_{y}}z_{j}\biggr)+
\frac12\biggl(\sum_{A_{i} \in L_{x}}z_{i}\biggr)\biggl(\sum_{A_{j} \in L_{xy}}z_{j}\biggr)+
\frac12\biggl(\sum_{A_i \in L_{y}}z_i\biggr)\biggl(\sum_{A_{j} \in L_{xy}}z_j\biggr)
\end{multline*}
This polynomial contains a term $z_uz_v$
whenever
$L$ is a tangent line for the pair $(A_u,A_v)$ or for the pair
$(A_v,A_u)$ (of the first or of the second kind,
and with coefficient~1 or $\frac12$, accordingly).  If we sum this
equation over all directed lines $L$, it follows by
Lemma~\ref{lemma:crucial} that every term $z_uz_v$ with $u\ne v$
appears with coefficient~2:
\begin{equation}\label{eq:1}
\sum_{L}P_L(z_{1}, \ldots,z_m)=\sum_{u<v}2z_u z_v=
(z_1+ \dotsb +z_{m})^2-(z_{1}^2+ \dots + z_{m}^2)
\end{equation}

Consider the system of linear equations
 $\sum_{A_{i} \in L_{x}}z_{i}=0$
and
 $\sum_{A_{i} \in L_{y}}z_{i}=0$, 
where
 $L=\overrightarrow{xy}$ varies over all directed lines
 determined by $P$.
 Add to this system the equation 
$z_{1}+ \dots +z_{m}=0$. There are $4\binom n 2+1$ equations in this system
and if $m > 4{\binom n 2}+1$, there must be a nontrivial solution.
However, it is easily seen that a nontrivial solution $(z_{1}, \ldots, z_{m})$
will result in a contradiction to (\ref{eq:1}). This is because the 
left-hand side of (\ref{eq:1}) vanishes, while the right-hand side 
equals $-(z_{1}^2+ \dots +z_{m}^2) \neq 0$. We conclude that 
$|F|=m \leq 4{\binom n2}+1$.
\end{proof}

We now show by a simple construction that Theorem \ref{theorem:n2} is tight 
apart from the multiplicative constant factor of $n^2$.
Fix three rays $r_{1}, r_{2}$, and $r_{3}$ emanating from the origin 
such that the angle between two rays is $120$ degrees. For each $i=1,2,3$, let 
$p^i_1, \ldots, p^i_n$ be $n$ points on $r_{i}$, indexed according to their
increasing distance from the origin.
Slightly perturb the points to get a set $P$ of $3n$ points in 
general position in the plane. For every $1 \leq j,k,l \leq n$ 
define 
$$
F_{jkl}=\{p^1_1, \ldots, p^1_j\} \cup \{p^2_1, \ldots, p^2_k\}
\cup \{p^3_1, \ldots, p^3_l\}.
$$
It can easily be checked that the collection of all $F_{jkl}$ such that 
$1 \leq j,k,l \leq n$ and $j+k+l=n+2$ is an anti-chain of convex
pseudo-discs of $P$. This collection consists of $\binom{n+1}2$
sets. 


\begin{thebibliography}{BRSSS04}

\bibitem[BRSSS04]{BRSSS04}
J. Balogh, O. Regev, C. Smyth, W. Steiger, and M. Szegedy,
{Long monotone paths in line arrangements}. 
{\it Discrete Comput. Geom.} {\bf 32} (2004), no. 2, 167--176. 

\bibitem[D98]{D98}
T. K. Dey,
{Improved bounds for planar $k$-sets and related problems}. 
{\it Discrete Comput. Geom.} {\bf 19} (1998), no. 3,
373--382.

\bibitem[ES88]{ES88}
P. H. Edelman and  M. E. Saks,  
{Combinatorial representation and convex dimension of convex geometries.}
{\it Order} {\bf 5} (1988), no. 1, 23--32.

\bibitem[E87]{Edelbook}
H. Edelsbrunner,
{\it Algorithms in Combinatorial Geometry},  EATCS Monographs on
Theoret. Comput. Sci., vol.~{10},
Springer-Verlag, Berlin, 1987.

\bibitem[GP72]{GP72} 
{R. L. Graham and H. O. Pollak},
{On embedding graphs in squashed cubes}. 
In {\it Proc. Conf. Graph Theory Appl.}, Western Michigan Univ.,
May 10--13, 1972,
ed.\ Y. Alavi, D. R. Lick, and A. T. White,
Lecture Notes in Mathematics, vol.~{303}, 
Springer-Verlag, Berlin, 1972,
pp. 99--110. 

\bibitem[T01]{T01}
G. T\'oth,
{Point sets with many $k$-sets},
{\it Discrete Comput. Geom.} {\bf 26} (2001) no. 2, 187--194.

\bibitem[T82]{T82} H. Tverberg,
{On the decomposition of $K_n$ into complete bipartite graphs.}
{\it J. Graph Theory} {\bf 6} (1982), no. 4, 493--494.

\end{thebibliography}
\end{document}